\documentclass[12pt,reqno]{amsart}

\usepackage{graphicx}
\graphicspath{ {./images/} }
\usepackage{bm}

\usepackage{color}

\usepackage{amssymb}

\usepackage{amsfonts}

\usepackage{amsmath}

\usepackage[all]{xy}

\newtheorem{theorem}{Theorem}[section]

\newtheorem{corollary}[theorem]{Corollary}

\newtheorem{lemma}[theorem]{Lemma}

\newtheorem{proposition}[theorem]{Proposition}

\newtheorem{Definition}[theorem]{Definition}

\newtheorem{Example}[theorem]{Example}

\newtheorem{Remark}[theorem]{Remark}

\newenvironment{remark}{\begin{Remark}\begin{em}}{\end{em}\end{Remark}}

\newcommand{\RomanNumeralCaps}[1]
    {\MakeUppercase{\romannumeral #1}}

\DeclareMathOperator{\tr}{tr}

\address{Miran Jeong \\ Department of Mathematics, Chungbuk National University, Cheongju 28644, Korea}
\email{jmr4006@chungbuk.ac.kr}

\address{Sejong Kim \\ Department of Mathematics, Chungbuk National University, Cheongju 28644, Korea}
\email{skim@chungbuk.ac.kr}

\address{Tin-Yau Tam \\ Department of Mathematics and Statistics, University of Nevada at Reno, NV 89557, USA}
\email{ttam@unr.edu}

\begin{document}

\author[Miran Jeong, Sejong Kim and Tin-Yau Tam]{Miran Jeong, Sejong Kim and Tin-Yau Tam}

\title{New weighted spectral geometric mean and quantum divergence}

\date{}
\maketitle

\begin{abstract}

A new class of weighted spectral geometric means has recently been introduced. In this paper, we present its inequalities in terms of the L\"owner order, operator norm, and trace. Moreover, we establish a log-majorization relationship between the new spectral geometric mean, and the R\'enyi relative operator entropy. We also give the quantum divergence of the quantity, given by the difference of trace values between the arithmetic mean and new spectral geometric mean. Finally, we study the barycenter that minimizes the weighted sum of quantum divergences for given variables.

\vspace{5mm}


\end{abstract}

\section{Introduction}

Kubo and Ando \cite{KA} introduced the concept of operator means for the set $\mathbb{P}$ of positive invertible operators on a complex Hilbert space $\mathcal{H}$, associated with operator monotone functions. Their  operator mean, denoted by $\sigma$, is defined for positive invertible operators $A$ and $B$, as
\begin{displaymath}
A \sigma B = A^{1/2} f(A^{-1/2} B A^{-1/2}) A^{1/2},
\end{displaymath}
where $f$ is an operator monotone function on $(0, \infty)$. A fundamental example of Kubo-Ando's operator means is the (metric) geometric mean defined using $f(x) = x^{t}$ for $t \in [0,1]$, and denoted by $A \sharp_{t} B$. In the finite-dimensional setting, $\mathbb{P}_{m}$, of $m \times m$ positive definite matrices, the geometric mean $A \sharp_{t} B$ has nice geometric meanings such as being the unique geodesic connecting $A$ and $B$ under the natural Riemannian trace distance
$$
d_{R}(A, B) = \Vert \log (A^{-1/2} B A^{-1/2}) \Vert_{F},
$$
where
$\Vert \cdot \Vert_{F}$ denotes the Forbenius norm.
In other words,
\begin{displaymath}
A \sharp_{t} B = \underset{X \in \mathbb{P}_{m}}{\arg \min} \, (1-t) d_{R}^{2}(X, A) + t d_{R}^{2}(X, B).
\end{displaymath}
 This characterization allows the extension of two-variable geometric means to multivariable means, such as the Cartan mean.

On the other hand, various operator means that do not satisfy monotonicity under the L\"owner partial order, known as non-Kubo-Ando means, have also been studied. A prominent example is the spectral geometric mean, whose weighted form is given in \cite{LL07}:
\begin{displaymath}
A \natural_{t} B = (A^{-1} \sharp B)^{t} A (A^{-1} \sharp B)^{t}, \quad t \in [0,1].
\end{displaymath}

Although the spectral geometric mean lacks monotonicity, it satisfies the in-betweenness property (a weaker version of monotonicity) with respect to the near-order
 \cite{DF24}. Many results comparing the geometric mean and the spectral geometric mean have been established \cite{DF, GK, GK24, GT22, Kim21}, including the following log-majorization relationship:
\begin{displaymath}
A \sharp_{t} B \prec_{\log} \exp ((1-t) \log A + t \log B) \prec_{\log} A \natural_{t} B.
\end{displaymath}
Recently, a new spectral geometric mean was introduced in \cite{DTV}:
\begin{displaymath}
F_{t}(A, B) := (A^{-1} \sharp_{t} B)^{1/2} A^{2-2t} (A^{-1} \sharp_{t} B)^{1/2}, \quad t \in [0,1].
\end{displaymath}
This is a  path connecting $A$ at $t = 0$ and $B$ at $t = 1$, with $F_{1/2}(A, B) = A \natural_{1/2} B$. It shares  properties with the spectral geometric mean $A \natural_{t} B$. In the context of positive invertible operators on $\mathcal{H}$, we derive interesting inequalities of this new mean in terms of L\"owner order and operator norm.
In the finite-dimensional setting $\mathbb{P}_{m}$, we establish a  trace inequality and a log-majorization relationship between the new spectral geometric mean and R\'{e}nyi relative operator entropy. These results allow us to define the quantity
\begin{displaymath}
\Phi(A, B) = \tr\, [A \nabla_{t} B - F_{t}(A, B)], \quad t \in [0,1]
\end{displaymath}
as a quantum divergence, analogous to the findings in \cite{GJK}. Moreover, we study the barycenter as a multi-variable mean that minimizes the weighted sum of quantum divergences for given variables.

\section{New weighted spectral geometric mean}

Let $B(\mathcal{H})$ be the Banach space of all bounded linear operators on a Hilbert space $\mathcal{H}$ over the complex field, and let $S(\mathcal{H}) \subset B(\mathcal{H})$ be the set of all self-adjoint operators.
Let $\mathbb{P} \subset S(\mathcal{H})$ be the open convex cone of all positive definite operators. For the finite-dimensional setting $\mathcal{H} = \mathbb{C}^{m}$, we denote by $\mathbb{H}_{m}$ and $\mathbb{P}_{m}$  the sets of all $m \times m$ Hermitian matrices and positive definite matrices, respectively.

The (metric) geometric mean of $A, B \in \mathbb{P}$ is defined as
\begin{displaymath}
A \sharp_{t} B := A^{1/2} (A^{-1/2} B A^{-1/2})^{t} A^{1/2}, \quad t \in [0,1].
\end{displaymath}
This mean represents the unique geodesic on $\mathbb{P}_{m}$ joining $A$ and $B$ under the natural Riemannian trace metric. Note that $A \sharp B = A \sharp_{1/2} B$ is a unique positive definite solution of Riccati equation
$$
X A^{-1} X = B.
$$ The following is the well-known arithmetic-geometric-harmonic mean inequalities:
\begin{displaymath}
A !_{t} B \leq A \sharp_{t} B \leq A \nabla_{t} B, , \quad t \in [0,1],
\end{displaymath}
where $A \nabla_{t} B = (1-t) A + t B$ and $A !_{t} B = [(1-t) A^{-1} + t B^{-1}]^{-1}$ are the arithmetic and harmonic means, respectively.

The spectral geometric mean of $A, B \in \mathbb{P}$ is defined as
\begin{displaymath}
A \natural_{t} B := (A^{-1} \sharp B)^{t} A (A^{-1} \sharp B)^{t}, \quad t \in [0,1].
\end{displaymath}
This mean has properties analogous to the (metric) geometric mean \cite{GK, GT22, Kim21}  and represents a geodesic on $\mathbb{P}$ joining $A$ and $B$ under the semi-metric $d(A,B) = 2 \Vert \log (A^{-1} \sharp B) \Vert$, where $\Vert \cdot \Vert$ denotes the operator norm. For simplicity, we write $A \natural B = A \natural_{\frac{1}{2}} B$.

Recently, a new weighted spectral geometric mean was introduced in \cite{DTV}:
\begin{displaymath}
F_{t}(A, B) = (A^{-1} \sharp_{t} B)^{\frac{1}{2}} A^{2-2t} (A^{-1} \sharp_{t} B)^{\frac{1}{2}}, \quad t\in [0,1].
\end{displaymath}
One can easily see that $F_{0}(A, B) = A$, $F_{1}(A, B) = B$, and $F_{\frac{1}{2}}(A, B) = A \natural B$. Moreover, $F_{t}(A, B)$ is a unique positive definite solution of the Riccati equation
\begin{displaymath}
(A^{-1} \sharp_{t} B)^{\frac{1}{2}} = A^{2t-2} \sharp X, , \quad t \in [0,1].
\end{displaymath}

The following properties of $F_{t}(A, B)$ were established  in \cite[Proposition 2.2]{DTV} for the finite-dimensional setting $\mathbb{P}_{m}$, and the proofs extend to $\mathbb{P}$.
\begin{proposition} \label{P:properties}
Let $A, B \in \mathbb{P}$. For all $t \in [0,1]$,
the following properties hold.
\begin{itemize}
  \item[(1)] $F_{t}(A, B) = A^{1-t} B^{t}$ if $A, B$ commute.
  \item[(2)] $F_{t}(a A, b B) = a^{1-t} b^{t} F_{t}(A, B)$ for any $a, b > 0$.
  \item[(3)] $F_{t}(U A U^{*}, U B U^{*}) = U F_{t}(A, B) U^{*}$ for any unitary operator $U$.
  \item[(4)] $F_{t}(A^{-1}, B^{-1}) = F_{t}(A, B)^{-1}$.
  \item[(5)] $2(A \nabla_{t} B^{-1})^{-\frac{1}{2}} - A^{2(t-1)} \leq F_{t}(A, B) \leq \left[ 2(A^{-1} \nabla_{t} B)^{-\frac{1}{2}} - A^{2(1-t)} \right]^{-1}$.
  \item[(6)] $\displaystyle \lim_{s \to 0} F_{t}(A^{s}, B^{s})^{1/s} = \exp ((1-t) \log A + t \log B)$.
%  \item[(6)] $\Vert F_{t}(A, B) \Vert \leq \Vert A \Vert^{1-t} \Vert B \Vert^{t} \quad$ for any $t \in [0,\frac{1}{2}]$, where $\Vert \cdot \Vert$ denotes an operator norm.
\end{itemize}
\end{proposition}

% (6) Note that for any $A, B \in \mathbb{P}$
% \begin{itemize}
%   \item[(i)] $\Vert A^{p} \Vert = \Vert A \Vert^{p}$ for any $p \geq 0$,
%   \item[(ii)] $\Vert A B \Vert = \Vert B A \Vert$,
%   \item[(iii)] $\Vert A B \Vert \leq \Vert A \Vert \cdot \Vert B \Vert$, and
%   \item[(iv)] $\Vert A \sharp_{t} B \Vert \leq \Vert A \Vert^{1-t} \Vert B \Vert^{t}$ for any $t \in [0,1]$.
% \end{itemize}
% Then
% \begin{displaymath}
% \begin{split}
% \Vert F_{t}(A, B) \Vert & = \Vert A^{1-t} (A^{-1} \sharp_{t} B) A^{1-t} \Vert = \Vert A^{1-2t} \sharp_{t} A^{1-t} B A^{1-t} \Vert \\
% & \leq \Vert A^{1-2t} \Vert^{1-t} \cdot \Vert A^{1-t} B A^{1-t} \Vert^{t} \\
% & \leq \Vert A \Vert^{(1-t)(1-2t)} \Vert A \Vert^{2t(1-t)} \Vert B \Vert^{t} = \Vert A \Vert^{1-t} \Vert B \Vert^{t}.
% \end{split}
% \end{displaymath}

\section{Operator inequalities and eigenvalue relationships}

The (metric) geometric mean $A \sharp_{t} B$ has a well-known property:
\begin{equation}\label{AH}
A \sharp_{t} B \leq I \Longrightarrow  A^{p} \sharp_{t} B^{p} \leq I \quad \mbox{for all}\ p \geq 1.
\end{equation}
This result, first discovered by Ando and Hiai \cite{AH},  has been generalized to the Cartan mean
for $\mathbb{P}_n$  \cite{Ya}  and the Karcher mean for $\mathbb{P}$ \cite{LL14}. In  the finite-dimensional case $\mathbb{P}_{m}$, the Ando-Hiai inequality can be expressed as  the following log-majorization relationship:
\begin{displaymath}
(A^{p} \sharp_{t} B^{p})^{1/p} \prec_{\log} (A^{q} \sharp_{t} B^{q})^{1/q},\quad 0 < q \leq p \leq 1.
\end{displaymath}

\begin{lemma} \cite[Lemma 2.7]{HK22} \label{L:ineq}
Let $S \in S(\mathcal{H})$ and $X \in \mathbb{P}$. Then $S X S \leq X$ implies $S \leq I$. Furthermore, $S X S = X$ for $S, X \in \mathbb{P}$ if and only if $S = I$.
\end{lemma}

Using this  result, we establish the following inequalities for the new weighted spectral geometric mean $F_{t}(A, B) $.

\begin{proposition} \label{P:Ineq}
Let $A, B \in \mathbb{P}$ and $t \in [0,1]$. Then the following statements hold.
\begin{enumerate}
\item
If $F_{t}(A^{-1}, B) \leq A^{2(t-1)}$, then $ A^{p} \sharp_{t} B^{p} \leq I$ for all $p \geq 1$.
\item $F_{t}(A, B) = A^{2-2t}$ if and only if $B = A^{\frac{1}{t}-1}$ for $t \in (0,1]$.
\end{enumerate}
\end{proposition}

\begin{proof}
Note that $(A \sharp_{t} B)^{\frac{1}{2}} A^{2t-2} (A \sharp_{t} B)^{\frac{1}{2}} = F_{t}(A^{-1}, B)  \leq A^{2(t-1)}$. From Lemma \ref{L:ineq},  we have $(A\sharp_tB)^{\frac 12}\le I$, and thus $A \sharp_{t} B \leq I$ for  $t \in [0,1]$. By the Ando-Hiai inequality \eqref{AH}, $A^{p} \sharp_{t} B^{p} \leq I$ for all $p \geq 1$.

Furthermore, $F_{t}(A, B) = A^{2-2t}$ if and only if $A^{-1} \sharp_{t} B = I$ by Lemma \ref{L:ineq}, which is equivalent to $(A^{\frac{1}{2}} B A^{\frac{1}{2}})^{t} = A$. This yields $B = A^{\frac{1}{t}-1}$ for $t \in (0,1]$.
\end{proof}

\begin{corollary}
Let $A, B \in \mathbb{P}$ and $t \in (0,1]$. Then $F_{t}(A, B) \leq A^{2-2t}$ for $A \leq I$ implies $B \leq A^{-1}$.
\end{corollary}

\begin{proof}
By Proposition \ref{P:Ineq}, $F_{t}(A, B) \leq A^{2-2t}$  implies $(A^{-1} \sharp_{t} B)^{\frac 12} \leq I$ but not $(A^{-1} \sharp_{t} B) \leq I$. Since $A \leq I$, we have
\begin{displaymath}
(A^{\frac{1}{2}} B A^{\frac{1}{2}})^{t} \leq A \leq I,
\end{displaymath}
which simplifies to $A^{\frac{1}{2}} B A^{\frac{1}{2}} \leq I$, that is, $B \leq A^{-1}$.
\end{proof}

\begin{theorem} \label{T:equivalence}
Let $A, B \in \mathbb{P}$. Then for $t \in \mathbb{R}$,
$$
F_{t}(A, B) \leq I \Longleftrightarrow  (A^{\frac{1}{2}} B A^{\frac{1}{2}})^{t} \leq A^{2t-1}.
$$
% Moreover, $B \leq A^{1-1/t}$ for $t \in (0,1]$ implies $F_{t}(A, B) \leq I$.
\end{theorem}

\begin{proof}
Note that $F_{t}(A, B) \leq I$ if and only if $A^{2(1-t)} \leq (A^{-1} \sharp_{t} B)^{-1}$, which is equivalent to $A^{2(t-1)} \geq A^{-1} \sharp_{t} B$. By taking congruence transformation via $A^{\frac{1}{2}}$, we obtain $A^{2t-1} \geq (A^{\frac{1}{2}} B A^{\frac{1}{2}})^{t}$.
\end{proof}

\begin{corollary} \label{C:upper-bdd}
Let $A, B \in \mathbb{P}$ such that $A \leq \alpha I$ and $B \leq \beta I$ for some $\alpha, \beta > 0$. Then $F_{t}(A, B) \leq \alpha^{1-t} \beta^{t} I$ for $t \in [0,1]$.
\end{corollary}

\begin{proof}
We first assume that $A, B \leq I$. Then we have $A^{\frac{1}{2}} B A^{\frac{1}{2}} \leq A$ since $B \leq I$, and
\begin{displaymath}
(A^{\frac{1}{2}} B A^{\frac{1}{2}})^{t} \leq A^{t} \leq A^{2t-1}.
\end{displaymath}
The first inequality follows from the L\"owner-Heinz inequality, and the last inequality follows from $A \leq I$ and $t \in [0,1]$. Thus, $F_{t}(A, B) \leq I$ by Theorem \ref{T:equivalence}.

For general $A \leq \alpha I$ and $B \leq \beta I$ for some $\alpha, \beta > 0$, let $A_{1} := \alpha^{-1} A \leq I$ and $B_{1} := \beta^{-1} B \leq I$, so $F_{t} (A_{1}, B_{1}) \le I$. Using the joint homogeneity of $F_{t}$, we have
\begin{displaymath}
F_{t} (A_{1}, B_{1}) = \frac{1}{\alpha^{1-t} \beta^{t}} F_{t} (A, B) \leq I.
\end{displaymath}
It yields the desired inequality.
\end{proof}

\begin{remark}
Since $A \leq \Vert A \Vert I$ and $B \leq \Vert B \Vert I$ for the operator norm $\Vert \cdot \Vert$, Corollary \ref{C:upper-bdd} yields $F_{t} (A, B) \leq \Vert A \Vert^{1-t} \Vert B \Vert^{t} I$ for any $t \in [0,1]$, and hence,
\begin{displaymath}
\Vert F_{t} (A, B) \Vert \leq \Vert A \Vert^{1-t} \Vert B \Vert^{t}.
\end{displaymath}
\end{remark}

Switching to the finite-dimensional setting $\mathbb{P}_{m}$, we now provide an inequality relating the trace of the new weighted spectral geometric mean to the traces of the individual matrices.

\begin{theorem} \label{T:tr-F}
Let $A, B \in \mathbb{P}_{m}$. Then $$\tr F_{t}(A, B) \leq (\tr A)^{1-t} (\tr B)^{t}, \quad t \in [0,1]. $$
\end{theorem}

\begin{proof}
For any $A, B \in \mathbb{P}_{m}$ and $t \in [0,1]$
\begin{equation} \label{E:f-tr}
\tr F_{t}(A, B) = \tr \left[ A^{1-t} (A^{-1} \sharp_{t} B) A^{1-t} \right] = \tr \left[ A^{\frac{1}{2} - t} (A^{\frac{1}{2}} B A^{\frac{1}{2}})^{t} A^{\frac{1}{2} - t} \right].
\end{equation}
By the Araki-Lieb-Thirring inequality in \cite{Ar},
\begin{displaymath}
\begin{split}
\tr F_{t}(A, B) & = \tr \left[ A^{\frac{1}{2} - t} (A^{\frac{1}{2}} B A^{\frac{1}{2}})^{t} A^{\frac{1}{2} - t} \right] \\
& \leq \tr \left[ (A^{\frac{1}{2t} - 1} \cdot A^{\frac{1}{2}} B A^{\frac{1}{2}} \cdot A^{\frac{1}{2t} - 1})^{t} \right] = \tr \left[ (A^{\frac{1-t}{2t}} B A^{\frac{1-t}{2t}})^{t} \right] \\
& \leq (1-t) \, \tr (A) + t \, \tr (B).
\end{split}
\end{displaymath}
The last inequality follows from \cite[Theorem 10]{BJL18}. Replacing $A, B$ by density matrices $\rho = \frac{A}{\tr A}, \sigma = \frac{B}{\tr B}$ in the above inequality, respectively, yields $\tr F_{t}(\rho, \sigma) \leq 1$, and using the joint homogeneity of $F_{t}$ in Proposition \ref{P:properties} (2) we obtain
\begin{displaymath}
\tr F_{t}(A, B) \leq (\tr A)^{1-t} (\tr B)^{t}.
\end{displaymath}
\end{proof}

For $t \in [0,1]$ and $z > 0$ the quantity
\begin{displaymath}
Q_{t,z}(A, B) = (A^{\frac{1-t}{2z}} B^{\frac{t}{z}} A^{\frac{1-t}{2z}})^{z}
\end{displaymath}
is the matrix version of R\'{e}nyi relative entropy. In particular, $Q_{t,t}(A, B)$ is known as the sandwiched R\'{e}nyi relative entropy \cite{WWY}. The quantity $Q_{t,z}$ can be defined for all $t \in \mathbb{R}$. The next theorem establishes a connection between the new weighted spectral geometric mean and R\'enyi relative entropy through log-majorization

\begin{theorem} \label{T:log-majorization}
Let $A, B \in \mathbb{P}_{m}$. Then
\begin{displaymath}
F_{t}(A, B) \prec_{\log} Q_{t,z}(A, B), \quad 0 < z \leq t \leq 1.
\end{displaymath}
\end{theorem}

\begin{proof}
Note that $Q_{t,z}(A, B)$ and $F_{t}(A, B)$ are homogeneous and are preserved by antisymmetric tensor power. Since
$$
\det F_{t}(A, B) = (\det A)^{1-t} (\det B)^{t} = \det (A^{\frac{1-t}{2z}} B^{\frac{t}{z}} A^{\frac{1-t}{2z}})^{z} =\det Q_{t,z}(A, B),
$$ it is sufficient to show that
$$
Q_{t,z}(A, B) \leq I \quad \Longrightarrow \quad F_{t}(A, B) \leq I.
$$
Assume that $Q_{t,z}(A, B) \leq I$. Then $B^{\frac{t}{z}} \leq A^{\frac{t-1}{z}}$ because $z > 0$.
Since $0 < \frac{z}{t} \leq 1$, $B \leq A^{1 - \frac{1}{t}}$ by the L\"owner-Heinz inequality.
This yields $A^{1/2} B A^{1/2} \leq A^{\frac{2t-1}{t}}$, which implies $(A^{1/2} B A^{1/2})^{t} \leq A^{2t-1}$ by the L\"owner-Heinz inequality.
By Theorem \ref{T:equivalence} we obtain $F_{t}(A, B) \leq I$.
\end{proof}

\begin{corollary} \label{C:log-majorization}
Let $A, B \in \mathbb{P}_{m}$ and $\frac{1}{2} \leq t \leq 1$. Then
\begin{displaymath}
F_{t}(A, B) \prec_{\log} Q_{t,t}(A, B) \prec_{\log} A \natural_{t} B \prec_{w \log} A \diamond_{t} B.
\end{displaymath}
\end{corollary}

\begin{proof}
The first log-majorization follows by Theorem \ref{T:log-majorization} when $z = t$, and the last weak log-majorization follows by \cite[Theorem 4.4]{GK24}.
From \cite[Corollary 8]{DF},
\begin{displaymath}
Q_{t,z}(A, B) \prec_{\log} A \natural_{t} B
\end{displaymath}
when $z \geq \max \{ t, 1-t \}$ for $t \in [0,1]$. Thus, the preceding log-majorization holds when $\frac{1}{2} \leq t \leq 1$ and $z \geq t$. The proof is complete by choosing $z = t$.
\end{proof}

\section{Quantum divergence and barycenter}

Bhatia, Gaubert and Jain \cite{BGJ} introduced a notion of quantum divergence on the Riemannian manifold $\mathbb{P}_{m}$. A map $\Phi : \mathbb{P}_m \times \mathbb{P}_m \to \mathbb{R}$ is called a \emph{quantum divergence} if it satisfies the following properties:
\begin{itemize}
\item[(I)] $\Phi(A, B) \geq 0$, and equality holds if and only if $A = B$;
\item[(II)] The derivative $D \Phi$ of $\Phi$ with respect to the second variable vanishes on the diagonal, that is,
\begin{displaymath}
D \Phi(A, B) |_{B=A} = 0;
\end{displaymath}
\item[(III)] The second derivative is non-negative on the diagonal, that is,
\begin{displaymath}
D^{2} \Phi(A, B) |_{B=A} [Y,Y] \geq 0
\end{displaymath}
for any Hermitian matrix $Y$.
\end{itemize}

Gan, Jeong and Kim showed in \cite{GJK} that the map
\begin{displaymath}
\Psi(A, B) = \tr (A \nabla_{t} B - A \natural_{t} B)
\end{displaymath}
is a quantum divergence for $t \in [0, \frac{1}{2}]$, and studied the barycenter minimizing the weighted sum of quantum divergences.
% We obtain similar results for new weighted spectral geometric mean.

\begin{theorem}\label{div}
The map $\Phi: \mathbb{P}_{m} \times \mathbb{P}_{m} \to \mathbb{R}$ defined by
\begin{equation} \label{E:q-divergence}
\Phi(A, B) = \tr [A \nabla_{t} B - F_{t}(A, B)]
\end{equation}
is a quantum divergence for $t \in [0,1]$.
\end{theorem}

\begin{proof} We verify the three conditions of quantum divergence.

From Theorem \ref{T:tr-F}, we have
$$
\tr F_{t}(A, B) \leq (\tr A)^{1-t} (\tr B)^{t} \leq (1-t) \tr A + t \tr B.
$$
Therefore,
$$\Phi(A, B)= \tr [A \nabla_{t} B - F_{t}(A, B)] \geq 0$$
with equality if and only if $A=B$.

To prove positive definiteness of $\Phi$, consider:

\begin{itemize}
  \item[(a)] Case 1: $t \in [0, \frac{1}{2}]$. If $\Phi(A, B) = 0$, then $(\tr A)^{1-t} (\tr B)^{t} = (1-t) \tr A + t \tr B$, and $\tr F_{t}(\rho, \sigma) = 1$ for positive definite density matrices $\rho = \frac{A}{\tr A}$ and $\sigma = \frac{B}{\tr B}$. By the definition of (metric) geometric mean on \eqref{E:f-tr}, we have
\begin{displaymath}
\begin{split}
\tr F_{t}(\rho, \sigma) & = \tr \left[ \rho^{\frac{1-2t}{2}} (\rho^{\frac{1}{2}} \sigma \rho^{\frac{1}{2}})^{t} \rho^{\frac{1-2t}{2}} \right] = \tr \left[ \rho^{1-2t} \left( (\rho^{\frac{1}{2}} \sigma \rho^{\frac{1}{2}})^{\frac{1}{2}} \right)^{2t} \right] \\
& \leq (1-2t) \tr \rho + 2t \tr (\rho^{\frac{1}{2}} \sigma \rho^{\frac{1}{2}})^{\frac{1}{2}} \\
& \leq 1.
\end{split}
\end{displaymath}
The first inequality follows from the fact that
\begin{displaymath}
\tr\, (A \sharp_{t} B) \leq \tr (A^{1-t} B^{t}) \leq (1-t) \tr A + t \tr B,\quad t \in [0,1],
\end{displaymath}
 and the second inequality follows from the property of quantum fidelity. Since $\tr F_{t}(\rho, \sigma) = 1$, we have
\begin{displaymath}
\tr\, (\rho^{\frac{1}{2}} \sigma \rho^{\frac{1}{2}})^{\frac{1}{2}} = 1.
\end{displaymath}
So $\tr A = \tr B$, and $\rho = \sigma$. Thus, $A = B$.

  \item[(b)] Case 2: $t \in [\frac{1}{2}, 1]$. Note % from Corollary \ref{C:log-majorization}
  that
\begin{displaymath}
F_{t}(A, B) \prec_{\log} Q_{t,t}(A, B) \prec_{\log} A \natural_{t} B \prec_{w \log} A \diamond_{t} B \leq A \nabla_{t} B,
\end{displaymath}
where the last inequality follows by \cite[Theorem 6]{BJL19}. So we have
\begin{displaymath}
\tr F_{t}(A, B) \leq \tr Q_{t,t}(A, B) \leq \tr (A \natural_{t} B) \leq \tr (A \diamond_{t} B) \leq \tr(A \nabla_{t} B).
\end{displaymath}
Thus, $\Phi(A, B) = 0$ implies $\tr (A \nabla_{t} B) = \tr (A \natural_{t} B) = \tr (A \diamond_{t} B)$. Since
\begin{displaymath}
\tr \left[ A \nabla_{t} B - A \diamond_{t} B \right] = t(1-t) \tr [(A + B) - 2 (A^{1/2} B A^{1/2})^{1/2}] = t(1-t) d_{W}^{2}(A, B),
\end{displaymath}
where $d_{W}$ denotes the Bures-Wasserstein distance, we obtain $A = B$.
\end{itemize}

\noindent (ii) One can see from \eqref{E:f-tr} that
\begin{equation*}
  D\Phi(A,B)(B)
  %= D (\tr(A \nabla_t B - A^{1/2-t} (A^{1/2} B A^{1/2})^{t} A^{1/2-t}))(B)
  = \tr (tI - A^{1/2-t} D(A^{1/2} B A^{1/2})^{t}(B) A^{1/2-t}).
\end{equation*}
Using the following integral representation:
\begin{equation*}
  A^t = \frac{\sin t\pi}{\pi} \int_{0}^{\infty} (\lambda A^{-1} + I)^{-1} \lambda^{t-1} d\lambda,\quad t \in (0,1),
\end{equation*}
we have
\begin{displaymath}
\begin{split}
  & D(A^{1/2} B A^{1/2})^{t}(B) \\
  &= \frac{\sin t\pi}{\pi} \int_{0}^{\infty} (\lambda A^{-1/2} B^{-1} A^{-1/2} + I)^{-1} A^{-1/2} B^{-2} A^{-1/2} (\lambda A^{-1/2} B^{-1} A^{-1/2} + I)^{-1} \lambda^{t} d\lambda.
  \end{split}
\end{displaymath}
Thus, we obtain
\begin{align*}
  D\Phi(A,B)(B)|_{B=A}
  &= \tr \left(tI - A^{1/2-t} \frac{\sin t\pi}{\pi} \int_{0}^{\infty} (\lambda A^{-2} + I)^{-1} A^{-3} (\lambda A^{-2} + I)^{-1} \lambda^{t} d\lambda A^{1/2-t} \right) \\
  &= \tr \left(tI - A^{1-t} \frac{\sin t\pi}{\pi} \int_{0}^{\infty} (\lambda I + A^{2})^{-2} \lambda^{t} d\lambda A^{1-t} \right) \\
  &= \tr\, (tI - tA^{2(1-t)+2t-2}) = 0,
\end{align*}
where the last equality follows from the fact that for $t \in (0,1)$
\begin{equation*}
  tA^{t-1} = D(A^{t}) = \frac{\sin t\pi}{\pi} \int_{0}^{\infty} (\lambda I + A)^{-2} \lambda^{t} d\lambda.
\end{equation*}

\noindent (iii) Let $Y \in \mathbb{H}_{m}$. Set $X := \lambda A^{-1/2} B^{-1} A^{-1/2} + I$.
From the above observation, one can see that
\begin{align*}
  D\Phi(A,B)(B)[Y]
  &= \tr\, (tY - A^{1/2-t} D(A^{1/2} B A^{1/2})^{t}(B)[Y] A^{1/2-t}) \\
  &= \tr \left( tY - A^{1/2-t} \frac{\sin t\pi}{\pi} \int_{0}^{\infty} X^{-1} A^{-1/2} B^{-1} Y B^{-1} A^{-1/2} X^{-1} \lambda^{t} d\lambda A^{1/2-t} \right).
\end{align*}
Then, we see that
\begin{align*}
  D^{2}&\Phi(A,B)(B,B)[Y,Y] \\
  = & - \tr \left( A^{1/2-t} \frac{\sin t\pi}{\pi} \int_{0}^{\infty} Z A^{-1/2} B^{-1} Y B^{-1} A^{-1/2} X^{-1} \lambda^{t} d\lambda A^{1/2-t} \right) \\
  & + \tr \left( A^{1/2-t} \frac{\sin t\pi}{\pi} \int_{0}^{\infty} X^{-1} A^{-1/2} B^{-1} Y B^{-1} Y B^{-1} A^{-1/2} X^{-1} \lambda^{t} d\lambda A^{1/2-t} \right) \\
  & + \tr \left( A^{1/2-t} \frac{\sin t\pi}{\pi} \int_{0}^{\infty} X^{-1} A^{-1/2} B^{-1} Y B^{-1} Y B^{-1} A^{-1/2} X^{-1} \lambda^{t} d\lambda A^{1/2-t} \right) \\
  & - \tr \left( A^{1/2-t} \frac{\sin t\pi}{\pi} \int_{0}^{\infty} X^{-1} A^{-1/2} B^{-1} Y B^{-1} A^{-1/2} Z \lambda^{t} d\lambda A^{1/2-t} \right),
\end{align*}
where $Z := D(X^{-1})(B)[Y]$.
%Note that $D(X^{-1})(B)[Y] = \lambda X^{-1} A^{-1/2} B^{-1} Y B^{-1} A^{-1/2} X^{-1}$, and thus, $D(X^{-1})(B)[Y]|_{B=A} = \lambda (\lambda A^{-2}+I)^{-1} A^{-3/2} Y A^{-3/2} (\lambda A^{-2}+I)^{-1}$.
We denote the four terms in the above summation as \RomanNumeralCaps{1}, \RomanNumeralCaps{2}, \RomanNumeralCaps{3} and \RomanNumeralCaps{4}, respectively.
If $B=A$, then using the fact that
\begin{displaymath}
D(X^{-1})(B)[Y]|_{B=A} = \lambda (\lambda A^{-2}+I)^{-1} A^{-3/2} Y A^{-3/2} (\lambda A^{-2}+I)^{-1}
\end{displaymath}
we obtain
\begin{align*}
  & \textrm{\RomanNumeralCaps{1}} = \textrm{\RomanNumeralCaps{4}} \\
  &= -\tr \left( A^{1-t} \frac{\sin t\pi}{\pi} \int_{0}^{\infty} (\lambda I + A^{2})^{-1} Y A^{-1/2} (\lambda I + A^{2})^{-1} A^{-1/2} Y (\lambda I + A^{2})^{-1} \lambda^{t+1} d\lambda A^{1-t} \right)  \\
  & \textrm{\RomanNumeralCaps{2}} = \textrm{\RomanNumeralCaps{3}} \\
  &= \tr \left( A^{1-t} \frac{\sin t\pi}{\pi} \int_{0}^{\infty} (\lambda I + A^{2})^{-1} Y A^{-1} Y (\lambda I + A^{2})^{-1} \lambda^{t} d\lambda A^{1-t} \right).
\end{align*}
Thus, we have
\begin{align*}
  & D^{2}\Phi(A,B)(B,B)[Y,Y]|_{B=A} \\
  &= 2 \tr \left( \frac{\sin t\pi}{\pi} \int_{0}^{\infty} A^{1-t} (\lambda I + A^{2})^{-1} Y A^{-1} Y (\lambda I + A^{2})^{-1} A^{1-t} \lambda^{t} d\lambda \right) \\
  & \quad -2 \tr \left( \frac{\sin t\pi}{\pi} \int_{0}^{\infty} A^{1-t} (\lambda I + A^{2})^{-1} Y A^{-1/2} (\lambda I + A^{2})^{-1} A^{-1/2} Y (\lambda I + A^{2})^{-1} A^{1-t} \lambda^{t+1} d\lambda \right) \\
  &= 2 \tr \left( \frac{\sin t\pi}{\pi} \int_{0}^{\infty} A^{1-t} (\lambda I + A^{2})^{-1} Y H Y (\lambda I + A^{2})^{-1} A^{1-t} \lambda^{t+1} d\lambda \right),
\end{align*}
where $H := A^{-1} - \lambda A^{-1/2} (\lambda I + A^{2})^{-1} A^{-1/2}$.
Since $I + \lambda^{-1} A^{2} > I$ for any $\lambda > 0$, $I - ( I + \lambda^{-1} A^{2} )^{-1} > 0$.
Applying the congruence transformation via $A^{-1/2}$ on both sides yields that
\begin{equation*}
  A^{-1} - \lambda A^{-1/2} (\lambda I + A^{2})^{-1} A^{-1/2}
  = A^{-1/2} \left\{ I - \left( I + \frac{1}{\lambda} A^{2} \right)^{-1} \right\} A^{-1/2} > 0.
\end{equation*}
Hence,  $D^{2}\Phi(A,B)(B,B)[Y,Y]|_{B=A} \geq 0$ for any $Y \in \mathbb{H}_{m}$.
\end{proof}

\begin{proposition} \label{P:invariance}
The quantum divergence $\Phi$ given by \eqref{E:q-divergence} is invariant under unitary congruence transformation and tensor product with a density matrix.
\end{proposition}

\begin{proof}
By Proposition \ref{P:properties} (3),
\begin{displaymath}
\Phi(U A U^{*}, U B U^{*}) = \tr [U (A \nabla_{t} B - F_{t}(A, B)) U^{*}] = \tr [A \nabla_{t} B - F_{t}(A, B)] = \Phi(A, B)
\end{displaymath}
for any unitary matrix $U$. Since
\begin{displaymath}
F_{t}(A \otimes C, B \otimes D) = F_{t}(A, B) \otimes F_{t}(C, D)
\end{displaymath}
for any $A, B, C, D \in \mathbb{P}_{m}$, the bi-linearity of the tensor product ensures
\begin{displaymath}
\Phi(A \otimes \rho, B \otimes \rho) = \tr [(A \nabla_{t} B - F_{t}(A, B)) \otimes \rho]
= \tr [A \nabla_{t} B - F_{t}(A, B)] \tr \rho = \Phi(A, B)
\end{displaymath}
for any density matrix $\rho$.
\end{proof}

\begin{theorem} \label{T:st-concavity}
For a given $A \in \mathbb{P}_{m}$ and $t \in (0,1)$, the map $f: \mathbb{P}_{m} \to \mathbb{R}$ defined by $f(X) = \tr F_{t}(A, X)$ is strictly concave.
\end{theorem}

\begin{proof}
From \eqref{E:f-tr}, $f(X) = \tr F_{t}(A, X) = \tr [A^{1-t} (A^{-1} \sharp_{t} X) A^{1-t}]$ for $X \in \mathbb{P}_{m}$ and $t \in (0,1)$.
Then for any $X, Y \in \mathbb{P}_{m}$ and $\lambda \in [0,1]$,
\begin{displaymath}
\begin{split}
f((1 - \lambda) X + \lambda Y) & = \tr [A^{1-t} (A^{-1} \sharp_{t} ((1 - \lambda) X + \lambda Y)) A^{1-t}] \\
& \geq \tr [A^{1-t} ((1 - \lambda) A^{-1} \sharp_{t} X + \lambda A^{-1} \sharp_{t} Y) A^{1-t}] \\
& = (1 - \lambda) \tr [A^{1-t} (A^{-1} \sharp_{t} X) A^{1-t}] + \lambda \tr [A^{1-t} (A^{-1} \sharp_{t} Y) A^{1-t}] \\
& = (1 - \lambda) f(X) + \lambda f(Y).
\end{split}
\end{displaymath}
The inequality follows from the joint concavity of (metric) geometric mean:
\begin{displaymath}
A \sharp_{t} ((1 - \lambda) X + \lambda Y)) \geq (1 - \lambda) A \sharp_{t} X + \lambda A \sharp_{t} Y.
\end{displaymath}
The equality of the preceding argument holds if and only if $X = Y$. Thus, the map $f$ is strictly concave.
\end{proof}

\begin{remark}
The data processing inequality is an information-theoretic principle stating that the information content of a signal cannot increase under a local physical operation.  For a quantum divergence $\Phi$, this means that for any completely positive trace-preserving map $\Psi$ and for $A, B \in \mathbb{P}_{m}$, the inequality
\begin{displaymath}
\Phi(\Psi(A), \Psi(B)) \leq \Phi(A, B)
\end{displaymath}
must hold.

According to \cite[Theorem 5.16]{Wo},  if a map $\Phi$ is jointly convex and invariant under unitary congruence transformation and tensor product with density matrices, then it satisfies the data processing inequality. From Theorem \ref{T:st-concavity}, the quantum divergence $\Phi$ is convex with respect to the second variable, but not necessarily with respect to the first variable. Despite this, $\Phi$ satisfies the data processing inequality when $t = 1/2$.
\end{remark}

Theorem \ref{T:st-concavity} ensures that the quantum divergence $\Phi$ given by \eqref{E:q-divergence} is strictly convex with respect to the second variable. Consequently, the minimization problem for a positive probability vector $\omega = (w_{1}, \dots, w_{n})$ and $A_{1}, \dots, A_{n} \in \mathbb{P}_{m}$
\begin{displaymath}
\underset{X \in \mathbb{P}_{m}}{\arg \min} \sum_{j=1}^{n} w_{j} \Phi(A_{j}, X)
\end{displaymath}
has a unique solution. We call it the barycenter of $A_{1}, \dots, A_{n}$ for the divergence $\Phi$, denoted by $\mathfrak{B}_{t}(\omega; A_{1}, \dots, A_{n})$.

\begin{theorem} \label{T:equation}
The barycenter $\mathfrak{B}_{t}(\omega; A_{1}, \dots, A_{n})$ of $A_{1}, \dots, A_{n} \in \mathbb{P}_{m}$ and $t\in (0,1)$  for the quantum divergence $\Phi$ is the unique positive definite solution of the equation
\begin{equation} \label{E:equation}
tI = \sum_{j=1}^{n} w_{j} \frac{\sin t \pi}{\pi} \int_{0}^{\infty} A_{j}^{-t} (\lambda A_{j}^{-1} + X)^{-2} A_{j}^{-t} \lambda^{t} d\lambda,
\end{equation}
where $\omega = (w_{1}, \dots, w_{n})$ is a positive probability vector.
\end{theorem}

\begin{proof} From the definition of the barycenter, $\mathfrak{B}_{t}(\omega; A_{1}, \dots, A_{n})$ minimizes the objective function
$$
F(X) = \sum_{j=1}^{n} w_{j} \Phi(A_{j}, X),
$$
where $\Phi(A_j, X)= \tr [A_j \nabla_{t} X - F_{t}(A_j, X)]$.
From the proof of (ii) in Theorem \ref{div}, the gradient of $F(X) $ with respect to $X$ is given by
\begin{displaymath}
\nabla F(X) = \sum_{j=1}^{n} w_{j} \left[ tI - A_{j}^{\frac{1}{2}-t} \frac{\sin t \pi}{\pi} \int_{0}^{\infty} (\lambda I + A_{j}^{\frac{1}{2}} X A_{j}^{\frac{1}{2}})^{-1} A_{j} (\lambda I + A_{j}^{\frac{1}{2}} X A_{j}^{\frac{1}{2}})^{-1} \lambda^{t} d\lambda A_{j}^{\frac{1}{2}-t} \right].
\end{displaymath}
Since the objective function $F(X)$ is strictly convex by Theorem \ref{T:st-concavity}, the barycenter $\mathfrak{B}_{t}(\omega; A_{1}, \dots, A_{n})$ coincides with the unique positive definition solution $X$ of the equation obtained by setting  $\nabla F(X)=0$. Thus,
\begin{displaymath}
\begin{split}
t I & = \sum_{j=1}^{n} w_{j} A_{j}^{\frac{1}{2}-t} \frac{\sin t \pi}{\pi} \int_{0}^{\infty} (\lambda I + A_{j}^{\frac{1}{2}} X A_{j}^{\frac{1}{2}})^{-1} A_{j} (\lambda I + A_{j}^{\frac{1}{2}} X A_{j}^{\frac{1}{2}})^{-1} \lambda^{t} d\lambda A_{j}^{\frac{1}{2}-t} \\
& = \sum_{j=1}^{n} w_{j} A_{j}^{-t} \frac{\sin t \pi}{\pi} \int_{0}^{\infty} (\lambda A_{j}^{-1} + X)^{-2} \lambda^{t} d\lambda A_{j}^{-t},
\end{split}
\end{displaymath}
since $A_{j}^{\frac{1}{2}} (\lambda I + A_{j}^{\frac{1}{2}} X A_{j}^{\frac{1}{2}})^{-1} A_{j}^{\frac{1}{2}} = (\lambda A_{j}^{-1} + X)^{-1}$ for each $i$.

The gradient of \( F(X) \) with respect to \( X \) is given by:
\[
\nabla F(X) = \sum_{j=1}^n w_j \left[ tI - \sin(t\pi) \int_0^\infty A_j^{-t} (\lambda A_j^{-1} + X)^{-2} \lambda^t \, d\lambda \right].
\]
Setting \( \nabla F(X) = 0 \) yields the necessary condition for the minimizer \( B_t(\omega; A_1, \dots, A_n) \), which satisfies:
\[
tI = \sum_{j=1}^n w_j \sin(t\pi) \int_0^\infty A_j^{-t} (\lambda A_j^{-1} + X)^{-2} \lambda^t \, d\lambda.
\]
Since \( F(X) \) is strictly convex (Theorem 4.3), this equation has a unique positive definite solution \( X \), which corresponds to the barycenter \( B_t(\omega; A_1, \dots, A_n) \).
\end{proof}

\begin{theorem}
The barycenter $\mathfrak{B}_{t}(\omega; A_{1}, \dots, A_{n})$ of $\mathbb{A} = (A_{1}, \dots, A_{n}) \in \mathbb{P}_{m}^{n}$ and  satisfies the following:
\begin{itemize}
  \item[(1)] if $A_{1}, \dots, A_{n}$ commute then $$\displaystyle \mathfrak{B}_{t}(\omega; \mathbb{A}) = \left( \sum_{j=1}^{n} w_{j} A_{j}^{1-t} \right)^{\frac{1}{1-t}}, \quad
t \in (0,1), $$
  where $\omega = (w_{1}, \dots, w_{n})$ is a positive probability vector.
  \item[(2)] The barycenter is invariant under permutation of the input matrices, that is, $$\mathfrak{B}_{t}(\omega_{\sigma}; \mathbb{A}_{\sigma}) = \mathfrak{B}_{t}(\omega; \mathbb{A}),$$ for any permutation $\sigma$.
  \item[(3)] The barycenter is equivariant under unitary congruence transformation, that is,
  $$\mathfrak{B}_{t}(\omega; U \mathbb{A} U^{*}) = U \mathfrak{B}_{t}(\omega; \mathbb{A}) U^{*},$$ for any unitary matrix $U$.
\end{itemize}
\end{theorem}

\begin{proof}
For (1) it is enough to show that $\displaystyle X = \left( \sum_{j=1}^{n} w_{j} A_{j}^{1-t} \right)^{\frac{1}{1-t}}$ is a unique solution of \eqref{E:equation} when $A_{1}, \dots, A_{n}$. Assuming that $X$ and $A_{j}$'s commute, the equation \eqref{E:equation} reduces to
\begin{displaymath}
\begin{split}
tI & = \sum_{j=1}^{n} w_{j} A_{j}^{1-t} \frac{\sin t \pi}{\pi} \int_{0}^{\infty} (\lambda I + A_{j} X)^{-2} \lambda^{t} d\lambda A_{j}^{1-t} \\
& = \sum_{j=1}^{n} w_{j} A_{j}^{1-t} t (A_{j} X)^{t-1} A_{j}^{1-t} \\
& = t X^{t-1} \sum_{j=1}^{n} w_{j} A_{j}^{1-t}.
\end{split}
\end{displaymath}
Solving for $X$, we get $\displaystyle X = \left( \sum_{j=1}^{n} w_{j} A_{j}^{1-t} \right)^{\frac{1}{1-t}}$.

Since the weights $\omega = (w_{1}, \dots, w_{n})$ and matrices $A_{1}, \dots, A_{n}$ can be permuted without affecting the sum in the barycenter equation, it follows that $\mathfrak{B}_{t}(\omega_{\sigma}; \mathbb{A}_{\sigma}) = \mathfrak{B}_{t}(\omega; \mathbb{A})$ for any permutation $\sigma$.

From Proposition \ref{P:invariance}, $F_t$ is  invariant under unitary congruence transformations. Hence we have (3), that is, $\mathfrak{B}_{t}(\omega; U \mathbb{A} U^{*}) = U \mathfrak{B}_{t}(\omega; \mathbb{A}) U^{*}$.

\end{proof}

\vspace{4mm}

\textbf{Acknowledgement} \\

%All authors contributed equally, and

There are no conflicts of interest associated with this work.
The research of Sejong Kim was conducted during the research year of Chungbuk National University in 2024 at the University of Nevada, Reno. His work is also supported by the National Research Foundation of Korea funded by the Ministry of Science and ICT under grant no. NRF-2022R1A2C4001306.

% The research of Sejong Kim was conducted during his sabbatical leave at the University of Nevada, Reno, in 2024. He extends his gratitude to Chungbuk National University, the National Research Foundation of Korea, and the Ministry of Science and ICT (MSIT) for their support under Grant No. NRF-2022R1A2C4001306.

%%%%%%%%%%%%%%%%%%%%%%%%%%%%%%%%%%%%%%%%%%%%%%%%%%

\end{document}